\documentclass[12pt]{article}
\usepackage{amsmath}
\usepackage{amssymb}
\usepackage{amsthm}
\usepackage{amsfonts}
\usepackage{array}
\usepackage{graphicx}
\usepackage{epstopdf}
\begin{document}
	\newtheorem{thm}{Theorem}
	\newtheorem{definition}{Definition}
	\newtheorem{lem}{Lemma}
	\newtheorem{Result}{Result}
	\newtheorem{cor}{Corollary}
	\newtheorem{rem}{Remark}
	\newtheorem{note}{Note}
	
	\begin{center}
		\Large{\bf  Nonsplit Domination Cover Pebbling Number for Some Class of Middle Graphs }
	\end{center}    
	\begin{center}
		A. Lourdusamy$^{1}$, I. Dhivviyanandam$^{2}$, Lian Mathew$^{3}$\\
		$^{1}$Department of Mathematics,\\
		St. Xavier's College (Autonomous), Palayamkottai\\
		Email: lourdusamy15@gmail.com; https://orcid.org/0000-0001-5961-358X.\\
		$^{2}$ Reg. No : 20211282091003, Department of Mathematics,\\
		St. Xavier's College (Autonomous), Palayamkottai\\
		Affiliated to Manonmaniam Sundaranar University, Abisekapatti-627012, Tamil Nadu, India,\\
		Email: divyanasj@gmail.com; https://orcid.org/0000-0002-3805-6638.\\
		$^{3}$Department of Mathematics\\
		CHRIST(Deemed to be university), Pune- Lavasa Campus, India.\\
		
		Email: lianmathew64@gmail.com; https://orcid.org/0000-0002-4926-7756.\\
	\end{center}        
	\begin{abstract}
		
		Let $G$ be a connected graph. A pebbling move is defined as taking two pebbles from one vertex and placing  one pebble to an adjacent vertex and throwing away the other pebble. The non-split domination cover pebbling number, $\psi_{ns}(G)$, of a graph $G$ is the minimum of pebbles that must be placed on $V(G)$ such that after a sequence of pebbling moves, the set of vertices with a pebble forms a non-split dominating set of $G$, regardless of the initial configuration of pebbles. We discuss some basic results, NP-completeness of non-split domination number, and determine $\psi_{ns}$ for some families of Middle graphs.
		
		
		
	\end{abstract}
	\section{Introduction}
	Lagarias and Saks were the first ones to introduce the concept of pebbling and  F.R.K. Chung\cite{1} used the concept in pebbling to solve a number theoretic conjecture. Then many others followed suit including Glenn Hulbert who published a survey of pebbling  variants\cite{2}. The subject of graph pebbling has seen massive growth after Hulbert's survey. In the past 30 years so many new variants in graph pebbling have been  developed which can be applied to the field of transportation, computer memory allocation, game theory, and the installation of mobile towers. \\
	
	\noindent Let us denote $G'$s vertex and edge sets as $V (G) $and $E (G) $, respectively. Consider a graph with a fixed number of pebbles at each vertex. One pebble is thrown away and the other is placed on an adjacent vertex when two pebbles are removed from a vertex. This  process is known as a pebble move. The pebbling number of a vertex $v$ in a graph $G$ is the smallest number $\pi(G, v)$, allowing us to shift a pebble to $v$ using a sequence of  pebbling move, regardless of how these pebbles are located on G's vertices. The pebbling number, $\pi(G)$, of a graph $G$ is the maximum of $\pi(G, v)$ over all the vertices $v$ of a graph. Considering the concept of cover pebbling \cite{7} and  non-split domination \cite{6} we develop a new concept, called the non-split domination cover pebbling  number of a graph, denoted by $\psi_{ns}(G)$. In paper\cite{7} ``The cover pebbling number, $\lambda(G)$ is defined
	as the minimum number of pebbles required such that given any initial configuration of at least $\lambda(G)$ pebbles, it is possible to make a series of pebbling
	moves to place at least one pebble on every vertex of $G$" and in \cite{6} The domination cover pebbling number is defined as `` the minimum
	number of pebbles required so that any initial configuration of pebbles can be
	shifted by a sequence of pebbling moves so that the set of vertices that
	contain pebbles form a dominating set $S$ of $G$".  Kulli, V.R et al. introduced the non-split domination number in\cite{6}. A dominating set $D$ of a graph $G=(V, E)$ is a non-split dominating set if the induced $<V-D>$ is connected. The non-split domination number $\gamma_{ns}(G)$ of $G$ is the minimum cardinality of a non-split dominating set. We developed the concept of  non-split domination cover pebbling deriving from the concept of cover pebbling and  non-split domination in graphs. Thus, we arrived the definition of  the non-split domination cover pebbling number,$\psi_{ns}(G)$, of a graph $G$ as the minimum of pebbles that must be placed on $V(G)$ such that after a sequence of pebbling moves, the set of vertices with a pebble forms a non-split dominating set of $G$, regardless of the initial configuration of pebbles. We discuss the basic results and determine $\psi_{ns}$ for  path graphs, wheel graphs,  cycle graphs, and fan graph
	
	\section{Preliminaries}
	For graph-theoretic terminologies, the reader can refer to\cite{4,5}.

	\begin{Result} \cite{6}
		\begin{itemize}
			\item[1]. The non-split domination number of a complete graph $K_n$ is $\gamma_{ns} (K_n)= 1$.
			\item[2]. The non-split domination number of a Wheel graph is $\gamma_{ns} (W_n)= 1$.
			\item[3]. The non-split domination number of a path  is $\gamma_{ns} (P_n)= n-2$.
			\item[4]. The non-split domination number of Cycle is $\gamma_{ns}(C_n) = n-2$.
		\end{itemize}
	\end{Result}
	
	\begin{Result}\cite{8}
		The domination cover pebbling number of the wheel is $\psi(W_n)= n-2$.
	\end{Result}
	
	\begin{Result}\cite{7}
		The cover pebbling number of path $P_n$ is $\gamma(P_n)=2^{n}-1$.
	\end{Result}
	
	\section{Results}
	
	\begin{thm}
		For a simple connected graph $G$, $\psi(G) \leq \psi_{ns}(G) \le \sigma(G)$. \end{thm}
	
	\begin{thm}
		For a complete graph $K_n$ on $n$ vertices, the nonsplit domination cover pebbling number is, $\psi_{ns}(K_n)=1$
	\end{thm}
	\begin{thm}
		The non-split domination cover pebbling number of a wheel graph $W_n$ is, $\psi_{ns}(W_n)=\psi(W_n).$
	\end{thm}
	\subsection{NP- Completeness of Nonsplit Domination Problem}
	The proof is by reduction from the known NP-complete problem 'domination'. The domination problem asks `for a given graph $G$ and an integer $k$, does the graph $G$ contains a domination set of cardinality at most $k$?'. 
	\begin{thm}
		The nonsplit domination problem is NP-complete.
	\end{thm}
	\begin{proof}
		Let $G$ be any graph. Construct $G^{*}$ as follows: Let the graph $G$ be on the first level and let the Path $P_2$ be on the second level. Join each vertex of graph $G$ to both the vertices of the path $P_2$ (see Figure 1). It is clear that the construction can be done in polynomial time. \\
		Now, let $G$ be a graph with domination number $k$. Then, $k+1$ vertices will form the non-split domination set for the graph $G^{*}$ iff $G$ has a dominating set with $k$ vertices. \\
		Suppose that the graph $G$ has a dominating set with cardinality $k$. Choose any one vertex from the path $P_2$. Then, it is straightforward to see that these $k+1$ vertices form a dominating set for the graph $G^{*}$. Moreover, if we remove the $k+1$ vertices and the edges incident to it, we get a connected graph. Thus, the non-split domination number of $G^{*}$ is $k+1$ whenever $G$ has a domination set with cardinality $k$.\\
		Conversely, assume that $k+1$ vertices form the non-split dominating set for the graph $G^{*}$. Then, by the definition $k+1$ vertices form a dominating set for $G^{*}$, and hence the domination set for graph $G$ has cardinality $k$. Thus, $G$ has a dominating set with $k$ vertices whenever $G^{*}$ has a non-split domination set with order $k+1$. 
	\end{proof}
	\begin{figure}[h!]
		\centering
		\includegraphics[width=8cm,height=5cm]{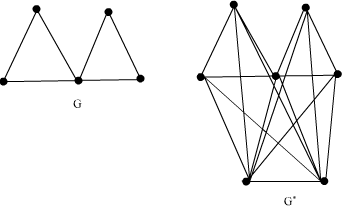}
		\caption{Construction of $G$ from $G^{*}$}
		\label{Sel}
	\end{figure}
	\subsection{Nonsplit Domination Cover Pebbling Number for the Middle Graph of Path }
	\begin{thm}
		The non-split domination cover pebbling number of the middle graph of the path  is, $\psi_{ns}(M(P_n))= 2^{n+1}-3$.
	\end{thm}
	\begin{proof}
		Let $V(P_n)=\{x_1, x_2, x_3,\cdots,x_n\}$ be the vertices of path $P_n$ and $y_1, y_2, \cdot, y_{n-1}$ be the added vertices corresponding to the edges $e_1, e_2,\cdots, e_{n-1}$ of $P_n$ to obtain $M(P_n)$.Then the total number of vertices is $2n-1$ and the edges are $3n-4$. Consider the non-split domination set $D=\{y_i| 1\leq i \leq n-1 \}$ with containing $n$ vertices. The vertices of $D$ dominate all the vertices in $M(P_n)$ and the induced subgraph $<V-D>$ is a connected graph. Placing $2^{n+1}-4$ pebbles on any one of the end vertices we could cover the non-split dominating set's vertices $x_{n}$ to $x_2$, if we place all the pebbles at $x_n$. And also we don't get the non-split domination set. Hence, the non-split domination cover pebbling number of $M(P_n)$ is $\psi_{ns}(M(P_n))\geq 2^{n+1}-3$.  \\
		
		\noindent Considering the configuration of $C$ with  $2^{n+1}-3$ on the vertices of $M(P_n)$, we prove the sufficient condition to cover the non-split dominating set. \\
		
		\noindent \textbf{Case 1: } Let the source vertex be $x_1$.\\
		If we place all the pebbles on $x_1$, to cover $x_n$ we need $ 2^{n}$ pebbles.  Thus, to cover the next furthest non-split dominating set is $x_{n-1}$ which requires $2^{n-1}$. Likewise, we get the series of Pebble distribution $2^{n} + 2^{n-1}+ 2^{n-2}+\cdots+ 2^2 +2^0$ to cover all the vertices of the non-split dominating set. Thus, we used $\sum_{k=2}^{n} 2^k +1= 2^{n+1}-3$ pebbles.\\
		
		\noindent  \textbf{Case 2: } Let the source vertex be $x_l,\ 1<l\leq n-1$.\\
		Let us place all the pebbles at $x_l$. If $l<\lfloor\frac{n}{2} \rfloor$, then we need $2^{(n-l) +1}-3$ pebbles to cover $x_l$ to $x_n$ of the non-split dominating set. Now to cover the remaining non-split dominating set we need $2^{l+1}-4$. Thus, we used $2^{(n-l) +1}+ 2^{l+1}-7 < 2^{n+1}-3$ pebbles. 
		\\
		
		\noindent \textbf{Case 3: } Let the source vertex be $y_1$.\\
		Let us consider the source vertex either any one of the non-split dominating vertices of $y_1$ or $y_n-1$. Consider all the pebbles placed on $y_1$. Then we need 4 pebbles to cover the adjacent vertices  $x_1$ and $x_2$. Then to cover the remaining  non-split dominating set of vertices we need $2^n-4$ pebbles. Thus, we used $2^n < 2^{n+1}-3$ pebbles.  Hence, $\psi_{ns}(M(P_n))= 2^{n+1}-3$.
	\end{proof}
	\subsection{Nonsplit Domination Cover Pebbling Number for The Middle Graph of Cycle Graphs}
	\begin{thm}
		The non-split domination cover pebbling number of the middle graph of the cycle  is,\\ $\psi_{ns}(M(C_n))= \begin{cases}  2 \sum_{k=0}^{\lceil \frac{n}{2}\rceil}2^{k}-8  & \  \ n \ {is\ odd}\\
			\\
			\sum_{k=1}^{\lfloor \frac{n}{2}\rfloor+1}2^{k}-8 +  \sum_{k=1}^{\lfloor \frac{n}{2}\rfloor}2^{k} 
			& \  \ n \ {is \ even}. \end{cases}$\\
	\end{thm}
	\begin{proof}
		Let $V(C_n)= \{x_1, x_2, x_3, \cdots, x_n\}$ and $y_1, y_2,\cdots, y_n$ be the inserted vertices corresponding to edges $e_1, e_2, \cdots, e_n$ of $C_n.$ to construct the middle graph of Cycle $M(C_n)$. Then the total number of vertices is $2n$ and the edges are $3n$. Let the non-split dominating set  $D=\{{y_i}\cup {x_1, x_2, x_3,\cdots, x_j}\}$ where $1\leq i,j\leq n$ and $x_j\ne N[y_i]$. Thus, $D$ dominates all the vertices of $M(C_n)$, and $<V-D>$ is connected. The total number of vertices in $D$ is $n-1$.\\

		\noindent \textbf{Case  1:} When $n$ is odd.\\
		Without loss of generality, Let $D=\{y_n, x_2, x_3,\cdots, x_{n-1}\}.$
		Placing $ 2 \sum_{k=0}^{\lceil \frac{n}{2}\rceil}2^{k}-7$ pebbles on the source vertex $x_1$ we can not put one pebble each on all the vertices of $D$. Hence, $\psi_{ns}(M(C_n))\geq  2 \sum_{k=0}^{\lceil \frac{n}{2}\rceil}2^{k}-8,$ when n is odd.
		
		\noindent Distribution of  $2 \sum_{k=0}^{\lceil \frac{n}{2}\rceil}2^{k}-8$ pebbles on the configuration of $C$, we cover all the vertices of $D$. Now we prove the sufficient condition for $M(C_n)$, when $n$ is odd.\\
		
		\noindent \textbf{Case 1.1:} Let the source vertex be $x_1$.\\
		Using the \textbf{Theorem 5} we can cover the non-split dominating set of $D$ from $x_2$ to $x_{\lceil\frac{n}{2}\rceil}$. The total number of pebbles used are $2^{\lceil\frac{n}{2}\rceil +1} -4$. Similarly, to cover the remaining non-split dominating set  \{$y_n, x_{n-1},\cdots,x_{\lceil\frac{n}{2}\rceil+1}$ we use $2^{\lceil\frac{n}{2}\rceil+1} -6$ pebbles. Hence we have spent total $2^{\lceil\frac{n}{2}\rceil+2}-10=  2 \sum_{k=0}^{\lceil \frac{n}{2}\rceil}2^{k}-8$ pebbles. 
		\noindent \textbf{Case 1.2:} Let the source vertex be any one of the vertices of  $y_i$, $1\leq i \leq n$.
		Without loss of generality, let the source vertex be $y_n$. 	Using the \textbf{Theorem 5} we can cover the non-split dominating set of $D$ from $x_2$ to $x_{\lceil\frac{n}{2}\rceil}$. The total number of pebbles used are $2^{\lceil\frac{n}{2}\rceil +1} -4$. Similarly, to cover the remaining non-split dominating set  \{$y_n, x_{n-1},\cdots,x_{\lceil\frac{n}{2}\rceil+1}$ we use $2^{\lceil\frac{n}{2}\rceil} -3$ pebbles. Hence we have spent total $3(2^{\lceil\frac{n}{2}\rceil})-7<  2 \sum_{k=0}^{\lceil \frac{n}{2}\rceil}2^{k}-8$ pebbles.  Hence,   $\psi_{ns}(M(C_n))\leq  2 \sum_{k=0}^{\lceil \frac{n}{2}\rceil}2^{k}-8$, When n is odd.\\
		
		\noindent \textbf{Case 2:} When $n$ is even.\\
		Without loss of generality, Let $D=\{y_n, x_2, x_3,\cdots, x_{n-1}\}.$
		Placing $\sum_{k=1}^{\lfloor \frac{n}{2}\rfloor+1}2^{k}-7 +  \sum_{k=1}^{\lfloor \frac{n}{2}\rfloor}2^{k}$ pebbles on the source vertex $x_1$ we can not put one pebble each on all the vertices of $D$. Hence, $\psi_{ns}(M(C_n))\geq  \sum_{k=1}^{\lfloor \frac{n}{2}\rfloor+1}2^{k}-8 +  \sum_{k=1}^{\lfloor \frac{n}{2}\rfloor}2^{k}$, When n is even.
		
		\noindent Distribution of  $\sum_{k=1}^{\lfloor \frac{n}{2}\rfloor+1}2^{k}-8 +  \sum_{k=1}^{\lfloor \frac{n}{2}\rfloor}2^{k}$ pebbles on the configuration of $C$, we cover all the vertices of $D$. Now we prove the sufficient condition for $M(C_n)$, when $n$ is even.\\
		
		\noindent \textbf{Case 2.1:} Let the source vertex be $x_1$.\\
		Using the \textbf{Theorem 5} we can cover the non-split dominating set of $D$ from $x_2$ to $x_{\lceil\frac{n}{2}\rceil+1}$. The total number of pebbles used are $2^{\lceil\frac{n}{2}\rceil+2} -4$. Similarly, to cover the remaining non-split dominating set  \{$y_n, x_{n-1},\cdots,x_{\lceil\frac{n}{2}\rceil+2}$ we use $2^{\lceil\frac{n}{2}\rceil+1} -6$ pebbles. Hence we have spent total $3(2^{\lceil\frac{n}{2}\rceil+1})-10= \sum_{k=1}^{\lfloor \frac{n}{2}\rfloor+1}2^{k}-8 +  \sum_{k=1}^{\lfloor \frac{n}{2}\rfloor}2^{k} $ pebbles. 
		\noindent \textbf{Case 2.2:} Let the source vertex be any one of the vertices of $y_i$, $1\leq i \leq n$.
		Without loss of generality, let the source vertex be $y_n$. 	Using the \textbf{Theorem 5} we can cover the non-split dominating set of $D$ from $x_2$ to $x_{\lceil\frac{n}{2}\rceil+1}$. The total number of pebbles used are $2^{\lceil\frac{n}{2}\rceil +2} -4$. Similarly, to cover the remaining non-split dominating set  \{$y_n, x_{n-1},\cdots,x_{\lceil\frac{n}{2}\rceil+2}$ we use $2^{\lceil\frac{n}{2}\rceil} -3$ pebbles. Hence we have spent total $5(2^{\lceil\frac{n}{2}\rceil})-7<  \sum_{k=1}^{\lfloor \frac{n}{2}\rfloor+1}2^{k}-8 +  \sum_{k=1}^{\lfloor \frac{n}{2}\rfloor}2^{k}$ pebbles.  Hence,   $\psi_{ns}(M(C_n))\leq  \sum_{k=1}^{\lfloor \frac{n}{2}\rfloor+1}2^{k}-8 +  \sum_{k=1}^{\lfloor \frac{n}{2}\rfloor}2^{k}$, when n is even. Hence proved.\\

	\end{proof}
	\subsection{Non-split Domination Cover Pebbling Number for The Middle Graph of Wheel Graphs}
	\begin{thm}
		The non-split domination cover pebbling number of the middle graph of the wheel  is,\\ $\psi_{ns}(M(W_n))= \begin{cases}  (\lfloor \frac{n}{2}\rfloor)8 +6 & \  \ n \ {is\ odd}\\
			\\
			(\lfloor \frac{n}{2}\rfloor)8 +  10
			& \  \ n \ {is \ even}. \end{cases}$\\
	\end{thm}
	\begin{proof}
		Let $V(W_n)= \{x_0, x_1, x_2, x_3, \cdots, x_{n-1}\}$ be the vertices of $W_n$ and $x_0$ be the center vertices of $W_n$. Let $y_1, y_2,\cdots, y_{n-1}$ be the inserted vertices corresponding to edges $v_0v_i$ where $1\leq i\leq n-1$ and $a_1, a_2, \cdots,a_{n-2}$ be the inserted vertices on the edges $x_jx_{j+1} $ where $1\leq j\leq n-2$ and $a_{n-1}$ lies in $x_{n-1}x_1$. Also the total number of edges in the $M(W_n)$ is $3(n-1)+1$.\\

		\noindent \textbf{Case  1:} When $n$ is even. Here $d(y_i)=n+3 \ (1\leq i\leq n-1)$\\
		\noindent \textbf{Sub case 1.1:} Suppose $i$ is odd.\\
		Consider the set $D=\{y_i, a_{i+1}, a_{i+3}, \cdots, a_{n-2} a_{i-2}, a_{i-5}\cdots, a_1\}$ where $a_j \ne N(y_i)$ and $j=1,3,\cdots, i-2, i+1, i+3, \cdots, n-1$ be a non-split dominating set and $<V_D>$ is connected which is the minimum Non-split dominating set of $M(W_n)$. If we place 	$(\lfloor \frac{n}{2}\rfloor)8 +  9$ pebbles on $x_1$, we can not cover all the vertices of $D$. Hence we require $	(\lfloor \frac{n}{2}\rfloor)8 +  10$ pebbles to cover $D$. Hence  $\psi_{ns}(M(W_n))\geq	(\lfloor \frac{n}{2}\rfloor)8 +  10$, When $n$ is even.

		\noindent Distribution of  	$(\lfloor \frac{n}{2}\rfloor)8 +  10$ pebbles on the configuration of $C$, we cover all the vertices of $D$. Now we prove the sufficient condition for $M(W_n)$, when $n$ is even.\\
		
		\noindent \textbf{Subcase 1.2:} Let the source vertex be $a_k$, where $k$ is not in $D$.\\
		Without loss of generality, let the source vertex be $a_1$. Since there will be 	$\lfloor \frac{n}{2}\rfloor -3$ vertices having the distance 3 and 2 vertices having the distance 2 from the source vertex then we need $(\lfloor \frac{n}{2}\rfloor)8 +  6 < (\lfloor \frac{n}{2}\rfloor)8 +  10$ Pebbles.
		\noindent \textbf{Subcase 1.3:} Let the source vertex be $x_0$. \\
		Let  us place all the pebbles on $x_0$. Since all the dominating vertices are at the distance of 2 from the center except $y_i$, we need $2+ \lfloor \frac{n}{2}\rfloor)4$ pebbles to cover the non-split dominating set. Thus, $\psi_{ns}(M(W_n))\geq	(\lfloor \frac{n}{2}\rfloor)8 +  10$, When $n$ is even and $i$ is odd. 
		
		\noindent \textbf{Case 1.2:} When $i$ is even.\\
		Consider the dominating set $D=\{y_i,a_{i+1},a_{i+3},\cdots,a_{n-1},a_{i-2},a_{i-4},\cdots,a_2\}$ where $a_j\ne N[y_i]$ and $ j= 2,4,\cdots,i-2,i+1,i+3,\cdots, n-1$ be the minimum dominating set and $<V_D>$ is connected. Now to prove the Non-split domination pebbling number of $M(W_n)$, when n is even and $i$ is even, we can follow the same method of \textbf{Case 1, Subcase 1.1, 1.2 and 1.3.}.
		
		\noindent \textbf{Case 2:} When $n$ is odd.\\
		Here $d(y_i)=n+3 \ (1\leq i\leq n-1)$\\
		\noindent \textbf{Subcase 2.1:} Suppose $i$ is odd.\\
		Consider the set $D=\{y_i, a_{i+1}, a_{i+3}, \cdots, a_{n-1} a_{i-2}, a_{i-4}\cdots, a_1\}$ where $a_j \ne N(y_i)$ and $j=1,3,\cdots, i-2, i+1, i+3, \cdots, n-1$ be a non-split dominating set and $<V_D>$ is connected which is the minimum Non-split dominating set of $M(W_n)$. If we place 	$(\lfloor \frac{n}{2}\rfloor)8 +  5$ pebbles on $x_1$, we can not cover all the vertices of $D$. Hence we require $	(\lfloor \frac{n}{2}\rfloor)8 +  6$ pebbles to cover $D$. Hence  $\psi_{ns}(M(W_n))\geq	(\lfloor \frac{n}{2}\rfloor)8 +  6$, When $n$ is odd.

		\noindent Distribution of  	$(\lfloor \frac{n}{2}\rfloor)8 +  6$ pebbles on the configuration of $C$, we cover all the vertices of $D$. Now we prove the sufficient condition for $M(W_n)$, when $n$ is odd.\\
		
		\noindent \textbf{Subcase 1.2:} Let the source vertex be $a_k$, where $k$ is not in $D$.\\
		Without loss of generality, let the source vertex be $x_1$. Since there will be 	$\lfloor \frac{n}{2}\rfloor -2$ vertices having the distance 3, one vertex at the distance of two and one vertex is adjacent to the source vertex, then we require $ (\lfloor \frac{n}{2}\rfloor)8 + 6$ pebbles.\\
		
		\noindent \textbf{Subcase 1.3:} Let the source vertex be $x_0$. \\
		Let  us place all the pebbles on $x_0$. Since all the dominating vertices are at the distance of 2 from the center except $y_i$, we need $2+ \lfloor \frac{n}{2}\rfloor)4$ pebbles to cover the non-split dominating set.
		\noindent \textbf{Subcase 1.4:} Let the source vertex be any one of the vertices of  $y_i$ where $i$ is odd. \\
		Let  us place all the pebbles on $y_1$. Since all the dominating vertices are at the distance of 2 from  $y_1$   then we need  $1+ \lfloor \frac{n}{2}\rfloor)4$ pebbles to cover the non-split dominating set. Thus, $\psi_{ns}(M(W_n))=(\lfloor \frac{n}{2}\rfloor)8 +  6$, When $n$ is even and $i$ is odd. 
		
		\noindent \textbf{Case 1.2:} When $i$ is even.\\
		Consider the dominating set $D=\{y_i,a_{i+1},a_{i+3},\cdots,a_{n-1},a_{i-2},a_{i-4},\cdots,a_2\}$ where $a_j\ne N[y_i]$ and $ j= 2,4,\cdots,i-2,i+1,i+3,\cdots, n-1$ be the minimum dominating set and $<V_D>$ is connected. Now to prove the Non-split domination pebbling number of $M(W_n)$, when n is odd and $i$ is even, we can follow the same method of \textbf{Case 2, Subcase 2.1, 2.2, 2.3 and 1.4.}. Thus proved.
		Hence, the non-split domination cover pebbling number of $M(W_n) =(\lfloor \frac{n}{2}\rfloor)8 +6$. 
		
		\subsection{Non-split Domination Cover Pebbling Number for The Middle Graph of Fan Graphs}
		\begin{thm}
			The non-split domination cover pebbling number of the middle graph of the fan  is,\\ $\psi_{ns}(M(F_n))= \begin{cases}  (\lceil \frac{n}{2}\rceil-1)8 +6 & \  \ n \ {is\ odd}\\
				\\
				(\lfloor \frac{n}{2}\rfloor-2)8 +  6
				& \  \ n \ {is \ even}. \end{cases}$\\
		\end{thm}
		\begin{proof}
			It follows from Theorem 7.
		\end{proof}
		
	\end{proof}

\end{document}